\documentclass[reqno]{amsart}[11pt]
\usepackage{amsmath, amsthm, amssymb, color}
\usepackage{graphicx}
\usepackage[mathscr]{euscript}
\usepackage{latexsym}
\usepackage{hyperref}
\usepackage{xcolor}
\numberwithin{equation}{section}
\newtheorem{Theorem}{Theorem}[section]
\newtheorem{Lemma}{Lemma}[section]

\theoremstyle{definition}
\newtheorem{Definition}{Definition}[section]
\newtheorem{Remark}{Remark}[section]

\newcounter{RomanNumber}

\def\be{\begin{equation}}
\def\en{\end{equation}}
\def\bs{\begin{split}}
\def\es{\end{split}}

\allowdisplaybreaks

\title[Global Stability and Non--Vanishing  Vacuum States]
{Global Stability and Non--Vanishing  Vacuum States of 3D
Compressible Navier--Stokes Equations}

\author{Guochun Wu}
\address{Guochun Wu \newline Fujian Province University Key Laboratory of Computational Science, School of Mathematical Sciences, Huaqiao University, Quanzhou 362021, P.R. China.}
\email{guochunwu@126.com}
\author{Lei Yao}
\address{Lei Yao \newline  School of Mathematics and Statistics, Northwestern Polytechnical University, Xi'an 710129, P.R. China.
 School of Mathematics and Center for Nonlinear Studies, Northwest University, Xi'an 710127, P.R. China.}
\email{yaolei1056@hotmail.com, leiyao@nwu.edu.cn}
\author{Yinghui Zhang*}
\address{Yinghui Zhang \newline Center for Applied Mathematics of Guangxi, Guangxi Normal University, Guilin, Guangxi 541004, P.R.
China} \email{yinghuizhang@mailbox.gxnu.edu.cn}

\thanks{Corresponding author: yinghuizhang@mailbox.gxnu.edu.cn.}
\bigbreak

\date{\today}
\usepackage{hyperref}

\begin{document}
\begin{abstract}
We investigate global stability and non--vanishing  vacuum
states of large solutions to the compressible Navier--Stokes
equations on the torus $\mathbb{T}^3$, and the main purpose of this
work is three-fold: First, under the assumption that the density
$\rho({\bf{x}}, t)$ verifies $\sup_{t\geq 0}\|\rho(
t)\|_{L^\infty}\leq M$, it is shown that the solutions converge to
equilibrium state exponentially in $L^2$--norm. In contrast to
previous related works where the density has uniform positive lower
and upper bounds, this gives the first stability result for large
strong solutions of the 3D compressible Navier--Stokes equations in
the presence of vacuum. Second, by employing some new thoughts, we
also show that the density converges to its equilibrium state
exponentially in $L^\infty$--norm if additionally the initial
density $\rho_0({\bf{x}})$ satisfies $\inf_{{\bf{x}}\in
\mathbb{T}^3}\rho_0({\bf{x}})\geq c_0>0$. Finally, we prove that the
vacuum state will persist for any time provided that the initial density contains vacuum, which is different from
the previous
work of [H. L. Li et al., Commun. Math. Phys., 281 (2008),
401--444], where the authors showed that any vacuum state must
vanish within finite time for the free boundary problem of the 1D
compressible Navier--Stokes equations with density--dependent
viscosity $\mu(\rho)=\rho^\alpha$ with $\alpha>1/2$.
This phenomenon implies the different behaviors
for Navier--Stokes equations with different types of viscous effects, namely,
degenerate or not.
\end{abstract}

\maketitle

\section{Introduction }
\setcounter{equation}{0}
In this paper, we are concerned with the global stability and
non--vanishing of vacuum states of large solutions to the
compressible Navier--Stokes equations on the torus $\mathbb{T}^3$:
\begin{equation}\begin{cases}
 \rho_t+\text{div} (\rho {\bf u})=0,\\
(\rho {\bf u})_t+\text{div}(\rho {\bf u}\otimes{\bf u})+\nabla P(\rho)=\mu\Delta {\bf u}+(\mu+\lambda) \nabla \hbox{div} {\bf u}.\end{cases}   \label{1.1}
\end{equation}
Here     $\rho=\rho({\bf x},t)$ and ${\bf u}=(u^1({\bf
x},t),u^2({\bf x},t),u^3({\bf x},t))^T$   stand for  the density and
the velocity respectively, at position ${\bf x}\in \mathbb T^3$ and
time $t\geq 0$. The pressure $P(\rho)=\rho^\gamma$, where $\gamma>1$
is the specific heat ratio. The constants $\mu$ and $\lambda$ are
the shear viscosity and the bulk viscosity of the fluid satisfying
the physical hypothesis:
$$\mu>0\ \ \text{and}\ \ 2\mu+3\lambda\geq 0.$$
Finally, the system \eqref{1.1} is supplemented with the following
initial condition:
\begin{equation}(\rho,\rho{\bf u})|_{t=0}=(\rho_0,\overline{{\bf m}}_0)(\bf
x),~~~~~~x\in \mathbb{T}^3\label{1.2}.
\end{equation}
Without loss of
generality, the mean value of total initial mass
over $\mathbb{T}^3$ is taken to be one throughout this paper, i.e.,
\begin{equation}\label{1.3} \frac{1}{|\mathbb T^3|}\int_{\mathbb T^3}\rho_0({\bf x})\mathrm{d}{\bf x}=1.\end{equation}
\smallskip

\subsection{History of the problem and main motivation} To put our results into context, let us highlight some progress on the topics of
global well--posedness  and  stability for the multidimensional
compressible Navier--Stokes equations. The global well--posedness of
classical solutions in the whole space $\mathbb{R}^3$ was firstly established by
Matsumura and Nishida \cite{Matsumura27} provided that the initial
data are close to a non--vacuum equilibrium in $H^3$. With the help of the \textit{effective viscous flux},
Hoff \cite{Hoff1,Hoff2} proved the global existence of weak
solutions with discontinuous initial data, i.e., the initial density
is close to a positive constant in $L^2$ and $L^\infty$, and the
initial velocity is small in $L^2$ and bounded in $L^{2^N}$, $N=2,3$
is the dimension of space.
 Under the framework of Besov space, Danchin \cite{Dan1} investigated  existence and uniqueness of the global strong
 solutions under the hypothesis that the initial value are close to a non--vacuum equilibrium state, see also \cite{Dan3,CMZ}.
 When the initial density is allowed to vanish and the spatial measure of the set of vacuum can be arbitrarily large, Huang et al.
 \cite{HuangLi1}   proved  global existence and uniqueness of classical solutions with smooth initial data
 that are of small energy in whole space $\mathbb{R}^3$, see also \cite{LiXin}. For the existence of solutions with arbitrary initial data,
 the major breakthrough is due to Lions \cite{Lions}, where he used the renormalization skills introduced by DiPerna and Lions
 \cite{DiPerna1} to establish global weak solutions if $\gamma> 3N/(N+2)$.  Later, Feireisl et al. \cite{Feireisl} improved  Lions's result
 to the case $\gamma>\frac{N}{2}$. When the initial data are assumed to have some spherically symmetric or  axisymmetric properties, Jiang and Zhang
 \cite{Jiang-Zhang 2001, Jiang1} proved the existence of global weak solutions for any $\gamma>1$. Plotnikov and Weigant \cite{Plotnikov} obtained
 the global existence of weak solutions to the  isothermal compressible Navier--Stokes equations in dimension two under some additional assumptions.
 Desjardins \cite{Desjardins} studied the regularity of weak solutions for small time under periodic boundary
 conditions, and particularly showed that weak solutions in $\mathbb T^2$ turn out to be smooth as
long as the density remains bounded in $L^\infty(\mathbb T^2)$. Due
to  the possible concentration of finite kinetic energy in very
small domains, whether those results in \cite{Feireisl,Lions} still
hold true for the case $\gamma \in [1,\frac{N}{2}]$ remains an
outstanding open problem.  Recently, Hu \cite{Hu} considered the
Hausdorff dimension of concentration for the compressible
Navier--Stokes equations. If $\gamma \in [1,\frac{N}{2}]$, he proved
that except for a space--time set with a Hausdorff dimension of less
than or equal to $\Gamma(n)+1$ with
$$\Gamma(n)=\max\left\{\gamma(n),n-\frac{n\gamma}{\gamma(n)+1}\right\}\ \text{and}\ \gamma(n)=\frac{n(n-1)-n\gamma}{n-\gamma},$$
no concentration phenomenon occurs.

In addition to the global--in--time existence, large time behavior
of solutions is also an important topic in the mathematical theory
of the physical world. Under the smallness assumption on the initial
perturbation,  the readers can refer  to
\cite{Dan4,Duan1,GuoY,LiuWang,Matsumura3,Matsumura2,Ponce} and
references therein for large time behavior of global smooth
solutions to the compressible Navier--Stokes system.  Recently, He
et al. \cite{He} investigated the global stability of large strong
solutions to the 3D Cauchy problem. Under the hypothesis that the
density $\rho({\bf x},t)$ verifies $\inf_{{\bf x}\in\mathbb
R^3}\rho_0({\bf x})\ge c_0>0$ and $\sup_{t\ge
0}\|\rho(t)\|_{C^\alpha}\le M$ with arbitrarily small $\alpha$, they
established a new approach for the convergence of the solutions to
its associated equilibrium states with an explicit decay rate which
is the same as that of the heat equation for the case
$\mu>\frac{1}{2}\lambda$. The assumption $\sup_{t\ge
0}\|\rho(t)\|_{C^\alpha}\le M$ played an essential role to derive
the uniform positive lower bound of the density $\rho$ (See the
proof of Proposition 2.3 in \cite{He} for details).  As concerned
with long time behavior of large weak solutions, Feireisl and
Petzeltov\'a \cite{Feireisl2} first showed that any weak solution
converges to a fixed stationary state as time goes to infinity via
the weak convergence method. Under the assumptions that the density
is essentially bounded and has uniform in time positive lower bound,
Padula \cite{Padula} proved that weak solutions decay exponentially
to the equilibrium state in $L^2$--norm. With the help of the
operator $\mathcal B$ introduced by Bogovskii \cite{Bogovskii}, Fang
et al. \cite{Fang} removed the restriction on the uniform positive
lower bound of the density. Recently, Peng--Shi--Wu \cite{PSW}
improved those results of \cite{Fang, Padula} to the case that they
didn't need both upper and lower time--independent bounds of
density. Recently, Zhang et al. \cite{Zi} showed that global regular
solutions of the full compressible Navier--Stokes equations on the
torus $\mathbb T^3$ converge to equilibrium with exponential rate
provided that both the density $\rho $ and temperature $\theta$
possess uniform in time positive lower and upper bounds.
\par
Motivated by \cite{He} and \cite{LiLiXin}, the main purpose of this
paper is to investigate global stability and non--vanishing vacuum
states of large strong solutions to the compressible Navier--Stokes
equations on the torus $\mathbb{T}^3$. More precisely, we are
concerned with the following three problems:

\par

(i) Notice that the density has uniform positive lower and upper
bounds in \cite{He, Zi}. Therefore, an important and interesting
problem is: What about the stability of large strong solutions for
the 3D compressible Navier--Stokes equations in presence of vacuum?

(ii) As mentioned before, assumptions that initial density
$\rho_0(x)$ has uniform positive lower bound and  $\sup_{t\ge
0}\|\rho(t)\|_{C^\alpha}\le M$ with arbitrarily small $\alpha$ in
\cite{He} played an essential role to derive the uniform positive
lower bound of the density $\rho$. Therefore, the natural and
interesting problem is: Can we show that the solutions converge to
equilibrium state exponentially in $L^2$--norm under the assumption
that $\sup_{t\geq 0}\|\rho( t)\|_{L^\infty}\leq M$ only?

\par(iii) Li--Li--Xin \cite{LiLiXin} showed that any
vacuum state will not exist within finite time for the free boundary
problem of the 1D compressible Navier--Stokes equations with
density--dependent viscosity $\mu(\rho)=\rho^\alpha$ with
$\alpha>1/2$. However, whether this result holds true for
multidimensional case still remains an outstanding open problem.
Therefore, a natural and important problem is: Provided that the initial density
contains vacuum, whether the vacuum state
persists or not for the 3D compressible Navier--Stokes equations?

\bigskip

The main purpose of this article is  to give a clear answer to the
above three problems.

\smallskip

\subsection{Main results}
Throughout this paper, we assume that the initial data satisfy
\begin{equation}
0\leq  \rho_0\in W^{1,q}(\mathbb T^3),\  \text{for\ some } \ q\in(3,6]\ \text{and}\
{\bf u}_0\in H^{2}(\mathbb T^3).   \label{1.4}
\end{equation}
We use $C$ to denote a generic constant independent of time which
may vary in different places. If $X$ is a Banach space, we will
abbreviate the vector--valued space $X^3$ by $X$ for convenience.
As in \cite{Feireisl, Hoff1,Hoff2, Lions}, the \textit{effective viscous flux} $F$
and vorticity ${\bf w}$ are defined by
\begin{equation} F\overset{\text{def}}= (2\mu+\lambda)\text{div}{\bf u}-(P(\rho)-1)
\ \ \text{and} \ \ {\bf w}\overset{\text{def}}=\mathcal P {\bf u}, \label{1.5}\end{equation}
where $\mathcal P=I +\nabla (-\Delta)^{-1}\textrm{div}$ denotes  the
projection on the space of divergence--free vector fields.

Now, we are ready to state our results. To begin with, we introduce
the definition of strong solutions to the problem
\eqref{1.1}--\eqref{1.2}.

\begin{Definition}[Strong solutions] For  $T>0$, a pair of function $(\rho, {\bf u})$ is said to be a strong solution of the  problem \eqref{1.1}--\eqref{1.2} on $\mathbb T^3\times [0,T]$, if for some $q\in(3,6]$
\begin{equation}\begin{cases}
0\leq \rho\in C([0,T];W^{1,q}(\mathbb T^3)),\ \ \rho_t\in C([0,T];L^q(\mathbb T^3)), \\
{\bf u}\in C([0,T];H^{2}(\mathbb T^3))\cap L^2(0,T; W^{2,q}(\mathbb T^3)),\ \
{\bf u}_t\in L^2(0,T;H^1(\mathbb T^3)),\\
\sqrt{\rho}{\bf u}_t\in L^\infty(0,T;L^2(\mathbb T^3)),\\
\end{cases}   \label{1.6}
\end{equation}
and $(\rho,\bf u)$ satisfies \eqref{1.1} a.e. on $\mathbb T^3\times [0,T]$.
\label{d1}\end{Definition}

If the initial density contains vacuum, we have the following result
on stability of a strong solution in $L^2$--norm to the  problem
\eqref{1.1}--\eqref{1.2}.
\begin{Theorem}\label{1mainth} Assume that the initial data $(\rho_0,{\bf u}_0)$ satisfy
\eqref{1.3}--\eqref{1.4},
and $K:=\|(\rho_0-1)\|_{L^2(\mathbb T^3)}+\|\sqrt{\rho_0}({\bf u}_0-{\bf
m}_0)\|_{L^2(\mathbb T^3)}+\|\nabla{\bf u}_0\|_{L^2(\mathbb
T^3)}<+\infty$ with ${\bf m_0}=\int_{\mathbb T^3}\rho_0{\bf
u_0}({\bf x})\mathrm{d}{\bf x}$. Let $(\rho,{\bf u})$ be a global
strong solution to the problem \eqref{1.1}--\eqref{1.2} verifying that
\begin{equation}\label{1.7}\sup\limits_{t\ge 0}\|\rho(\cdot,t)\|_{L^\infty(\mathbb T^3)}\leq M,\end{equation}
for some positive constant $M$. Then, there exist two positive
constants $C_1>0$ and $\eta_1>0$, which are dependent on $M$ and
$K$, but independent of $t$, such that
\begin{equation}\label{1.8}\|(\rho-1)(\cdot,t)\|_{L^2(\mathbb T^3)}+\|\sqrt{\rho}({\bf u}-{\bf m}_0)(\cdot,t)\|_{L^2(\mathbb T^3)}+\|\nabla
{\bf{u}}\|_{L^2(\mathbb T^3)}\leq C_1\text{e}^{-\eta_1 t},
\end{equation}
for any $t\ge 0$.
\end{Theorem}

\smallskip
\begin{Remark} The question naturally arises whether the solution $(\rho,{\bf u})$ stated in Theorem \ref{1mainth} exists or not.
When the initial density is allowed to vanish and the spatial
measure of the set of vacuum can be arbitrarily large, Huang, Li and
Xin   \cite{HuangLi1}   established the global existence and
uniqueness of classical solutions  in whole space $\mathbb{R}^3$ if
the initial energy is small but the oscillations could be
arbitrarily large, see also \cite{LiXin}. One of  the key
ingredients in \cite{HuangLi1,LiXin} is to derive a
time--independent upper bound of the density. So, under the
assumption that the initial energy is small,  using the similar
arguments as that in \cite{HuangLi1,LiXin}, we can show that a
strong solution $(\rho,{\bf u})$  satisfying  \eqref{1.7} indeed
exists.

\end{Remark}

\smallskip
\begin{Remark} Compared to \cite{He, Zi} where the density has uniform positive lower
and upper bounds, this gives the first stability result for large
strong solutions of the 3D compressible Navier--Stokes equations in
the presence of vacuum.
\end{Remark}

\begin{Remark} It is interesting to make a comparison between Theorem \ref{1mainth}
and those of Peng--Wu--Shi \cite{PSW}, where the authors give global exponential stability of
finite energy weak solutions constructed by Lions and Feireisl etc.
More precisely, for general large data, by both using the
extra integrability of the density due to Lions and constructing a suitable Lyapunov
functional, Peng--Wu--Shi \cite{PSW} showed that
\begin{equation}\nonumber\int_\Omega(\rho|\mathbf{u}|^2 + G(\rho, \rho_s))dx \leq C \exp\{-Ct\},\end{equation} where
$\rho_s=\displaystyle\frac{1}{|\Omega|}\int_\Omega \rho_0 dx$, and $G(\rho, \rho_s):= \displaystyle\rho \int_{\rho}^{\rho_s}\frac{h^\gamma-\rho_s^\gamma}{h^2}dh$.
If the density $\rho$ has upper bound, it is easy to check that the above exponential decay estimate implies
\begin{equation}\nonumber\|\rho-\rho_s\|_{L^2}+\|\sqrt{\rho}\mathbf{u}\|_{L^2}\leq C \exp\{-Ct\}.\end{equation}
Compared to the exponential decay estimate in \eqref{1.8}, this gives no information for the large time behavior of $\|\nabla \mathbf{u}\|_{L^2}$.

\end{Remark}

\smallskip
\begin{Remark} Our methods can be applied to investigate global
stability of large strong solutions to full compressible
Navier--Stokes equations on $\mathbb{T}^3$. When the initial density
is allowed to vanish, i.e., $\rho_0(x)\geq 0$, we can prove global
exponential stability of strong solutions provided that the density
$\rho(x, t)$ verifies $\sup\limits_{t\ge
0}\|\rho(\cdot,t)\|_{L^\infty(\mathbb T^3)}\leq M$. This result will
be reported in our forthcoming paper \cite{WYZ}.
\end{Remark}
\smallskip

If the initial density possesses uniform positive lower bound, we
have the following result on the stability of the density in
$L^\infty$--norm to the problem \eqref{1.1}--\eqref{1.2}.

\begin{Theorem}\label{2mainth} Assume that all conditions  of Theorem \ref{1mainth} are in
force. If additionally $\inf\limits_{{\bf x}\in\mathbb
T^3}\rho_0({\bf x})\ge c_0>0$, then there exist two positive
constants $C_2>0$ and $\eta_2>0$, which are dependent on $c_0$, $M$
and $K$, but independent of $t$, such that
\begin{equation}\label{1.9}\|(\rho-1)(\cdot,t)\|_{L^\infty(\mathbb T^3)}\leq C_2\text{e}^{-\eta_2
t},
\end{equation}
for any $t\ge 0$.
\end{Theorem}

\smallskip
\begin{Remark}
It is worth mentioning that to prove Theorem \ref{2mainth},
we only assume that the density $\rho$ is bounded from above, while
the theory on global stability of large solutions developed in
\cite{He} requires the additional assumption $\sup_{t\ge
0}\|\rho(t)\|_{C^\alpha}\le M$ with arbitrarily small $\alpha$,
which plays an essential role in deriving the uniform positive lower
bound of $\rho$ in \cite{He} (See the proof of Proposition 2.3 in
\cite{He} for details).
\end{Remark}
\smallskip

\smallskip
\begin{Remark}
To prove Theorem \ref{2mainth}, the key ingredient is to get a
time--independent positive lower bound of the density $\rho$ (See
Lemma \ref{lemma4.1}). With the key time--independent positive upper
and lower bounds of the density $\rho$ in hand, we can modify the
methods of \cite{Zi} to obtain the exponential decay rates of higher--order
spatial derivatives of the solutions.
\end{Remark}
\smallskip

Provided that the vacuum states are present initially, we shall
prove that the vacuum states will not vanish for any time.

\begin{Theorem}\label{3mainth} Assume that all conditions of   Theorem \ref{1mainth} are in
force. If additionally $\inf\limits_{{\bf x}\in\mathbb
T^3}\rho_0({\bf x})=0$, then it holds that
\begin{equation}\label{1.10}\inf\limits_{{\bf x}\in\mathbb T^3}\rho({\bf
x},t)=0,
\end{equation}
for any  $t\ge 0$.
\end{Theorem}
\smallskip
\begin{Remark}
Theorem \ref{3mainth} implies that the vacuum state will persist
for any time provided that the initial density contains vacuum,
which is different from the
previous work of Li--Li--Xin \cite{LiLiXin}. Indeed, Li--Li--Xin
\cite{LiLiXin} showed that for any global entropy weak solution, any
(possibly existing) vacuum state must vanish within finite time for
the free boundary problem of the 1D isentropic compressible
Navier--Stokes equations with density--dependent viscosity:
\begin{equation}\begin{cases}\nonumber
 \rho_t+ (\rho { u})_x=0,\\
(\rho { u})_t+\text{div}(\rho { u}^2)+\nabla
\rho^\gamma-(\rho^\alpha u_x)_x=0,\end{cases}   \label{1.12-1D}
\end{equation}
where $\alpha>\frac{1}{2}$ and $\gamma\ge 1$.  Therefore, there arises a natural question whether any vacuum state shall vanish within finite time or vacuum state is  preserved for any time for the case $0\le\alpha\le \frac{1}{2}$.
\end{Remark}
\smallskip

\smallskip

\subsection{Outline of ideas}
We make some comments on  the main  ideas of the proof and explain
the main difficulties and techniques involved in the process. \par
The proof of Theorem \ref{1mainth} can be outlined as follows.
Firstly, we prove the exponential decay estimate of
$\|(\rho-1,\sqrt{\rho}({\bf u}-{\bf m}_0))(t)\|_{L^2(\mathbb T^3)}$.
Set
 ${\bf v}={\bf u}-{\bf m_0}$ with ${\bf
m_0}=\int_{\mathbb{T}^3}\rho_0{\bf u}_0({{\bf x}})\mathrm{d}{\bf x}.$  By making the basic energy estimate on the problem
\eqref{1.1}--\eqref{1.2},
one can  derive
 an energy-dissipation inequality of the form
\begin{equation}\label{1.11}\frac{\mathrm{d}}{\mathrm{d}t}\tilde{\mathcal E}(t)+\tilde{\mathcal D}(t)\leq 0,
 \end{equation}
where energy $\tilde{\mathcal E}(t)$ is equivalent to
$\|(\rho-1,\sqrt{\rho}{\bf v})\|_{L^2}^2$, and dissipation
$\tilde{\mathcal D}(t)$ is equivalent to $\|(\rho-1, \nabla
{\bf u}\|_{L^2}^2$. On the other hand, by making full use of
momentum equation and Poincar\'{e}'s inequality, it is clear that
$\|\nabla {\bf u}\|_{L^2}\geq C\|\sqrt{\rho}{\bf v}\|_{L^2}$, this
particularly implies that $\tilde{\mathcal D}(t)\geq
C\tilde{\mathcal E}(t)$. Consequently, the exponential decay
estimate of $\|(\rho-1,\sqrt{\rho}({\bf u}-{\bf
m}_0))(t)\|_{L^2(\mathbb T^3)}$ in Theorem \ref{1mainth} follows
from \eqref{1.11} immediately(See also the  Lyapunov--type energy inequality \eqref{3.15}). Secondly, we derive the exponential
decay estimate of $\|\nabla{\bf u}(t)\|_{L^2(\mathbb T^3)}$. To do
this, we make full use of good properties of the \textit{effective
viscous flux} $F$ to get the energy estimate \eqref{3.24}. To close
the estimate \eqref{3.24}, our main observation is that $\|\nabla
{\bf u}(\cdot, t)\|_{L^2}$ is sufficiently small for any large
enough $t$. With this key observation in hand, we can take a linear
combination of \eqref{3.15} and \eqref{3.24} to get the key
Lyapunov--type energy inequality \eqref{3.30}. Then, \eqref{3.30}
together with Gronwall's inequality implies the exponential decay
estimate of $\|\nabla{\bf u}(t)\|_{L^2(\mathbb T^3)}$ immediately.

To prove Theorem~\ref{2mainth} and Theorem~\ref{3mainth}, the key
ingredient is to establish the time--independent positive lower
bound of $\rho$. To achieve this goal, we will borrow some ideas
from \cite{Desjardins, Lions2} and make some key uniform estimate.
To see this, we first rewrite the mass conservation equation
$\eqref{1.1}_1$ in terms of $\log\rho$ (cf.~\eqref{4.2}). Then, by
defining
$H\overset{\triangle}=(2\mu+\lambda)\log\rho+\Delta^{-1}\text{div}(\rho{\bf
v})$, and fully using the momentum conservation equation
$\eqref{1.1}_2$ and Lagrangian coordinates, it is clear that along the particle trajectories $H$
satisfies \eqref{4.4}. Finally, we exploit some
delicate energy estimates for \eqref{4.4} to get key time--independent
negative lower bound of $H$:
\begin{equation}\nonumber H(t) \geq -C(\textrm{for} \ \textrm{some}\  \textrm{constant}\  C>0)
\end{equation}
holds for any large enough $t$. This together with \eqref{4.9} imply the time--independent
positive lower bound of $\rho$ immediately. The exponential decay estimate for $\|(\rho-1)\|_{L^\infty(\mathbb T^3)}$  is due to the damping mechanism of density. As a by--product, we finally
show that the vacuum states will not vanish for any time provided
that the vacuum states are present initially.

\bigskip

\par
The rest of the paper is organized as follows. In Section 2, we
recall some  elementary  facts and inequalities  that will be used
frequently in later analysis. Section 3 is devoted to proving
Theorem~\ref{1mainth}. We prove  Theorem~\ref{2mainth} and
Theorem~\ref{3mainth} in Section 4.

\bigskip\bigskip

\setcounter{equation}{0}
\section{Preliminaries}
\setcounter{equation}{0}
\smallskip
In this section, we list some  elementary but useful facts and
inequalities which will be used frequently in the sequel.

Set $\varrho\overset{\triangle}=\rho-1$ and   define  the potential energy density $G$  by
 \begin{equation}G(\rho)\overset{\triangle}=\rho\int^\rho_{1}\frac{P(s)-1}{s^2}\mathrm{d}s.\label{2.1}\end{equation}
The following lemma is concerned with the
 estimates about $P(\rho)-1$ and $G(\rho)$, see \cite{Fang}.
 \begin{Lemma}\label{lemma2.2} Let $\gamma>1$ be arbitrary fixed constants. Then we have
 \begin{equation}\nonumber P(\rho)-1\sim \varrho~~\text{and}~~G(\rho)\sim\varrho^2
 \end{equation}
 if $0\le\rho\le M$.
\end{Lemma}
\smallskip

 In virtue of \eqref{1.1}$_2$, one has
\begin{equation}\Delta F=\text{div}(\rho\dot{{\bf u}})\ \ \text{and} \ \ \mu\Delta  {\bf w}
=\mathcal P(\rho\dot {\bf u}),\label{2.2}\end{equation}
where $``\ \dot\ \ "$ denotes the material derivative which is defined by
$$\dot f\overset{\text{def}}=\partial_tf+{\bf u}\cdot\nabla f.$$
Applying the standard $L^p$--estimates of  elliptic systems to \eqref{2.2}, we have the following estimates.
\begin{Lemma}\label{lemma2.3} Let $(\rho,{\bf u})$ be a strong solution  to the problem \eqref{1.1}--\eqref{1.2}. Then for any $p\in(1,\infty)$, there exists a generic positive constant $C$ which
depends only on $\mu$, $\lambda$ and $p$ such that
\begin{equation}\|\nabla F\|_{L^p(\mathbb T^{3})}+\|\nabla^2{\bf w}\|_{L^p(\mathbb T^{3})}\leq C\|\rho\dot {\bf u}\|_{L^p(\mathbb T^{3})},\label{2.3}\end{equation}
and
\begin{equation}\|\nabla {\bf u}\|_{L^p(\mathbb T^{3})}\leq C\left(\|F\|_{L^p(\mathbb T^{3})}+\|\nabla {\bf w}\|_{L^p(\mathbb T^{3})}+\|(P(\rho)-1)\|_{L^p(\mathbb T^{3})}\right).\label{2.4}\end{equation}
\end{Lemma}
\begin{proof}Applying  the standard $L^p$--estimate of elliptic
systems to \eqref{2.2}, \eqref{2.3} follows immediately. Noticing
that $-\Delta {\bf u}=-\nabla \text{div} {\bf
u}+\nabla\times\nabla\times {\bf w}$, one has
$$\nabla {\bf u}=\nabla\Delta^{-1}\nabla \text{div} {\bf u}-\nabla\Delta^{-1}\nabla\times\nabla\times{\bf w},$$
where $\Delta^{-1}$ denotes the inverse Laplacian with zero mean
value on $\mathbb T^3$. Thus, it follows  the Marcinkiewicz
multiplier theorem (see \cite{Stein})  that
\begin{equation}\nonumber\begin {split}\|\nabla {\bf u}\|_{L^p(\mathbb T^{3})}&\leq C(\|\text{div} {\bf u}\|_{L^p(\mathbb T^{3})}+\|\nabla{\bf w}\|_{L^p(\mathbb T^{3})})
\\&\leq C(\|F\|_{L^p(\mathbb T^{3})}+\|\nabla{\bf w}\|_{L^p(\mathbb T^{3})}+\|(P(\rho)-1)\|_{L^p(\mathbb T^{3})})\end{split}\end{equation}
as claimed in \eqref{2.4}. The proof the lemma is completed.
\end{proof}

\bigskip\bigskip

\setcounter{equation}{0}
\section{Proof of Theorem \ref{1mainth}}
\setcounter{equation}{0}
\smallskip
In this section, we devote ourselves to proving Theorem
\ref{1mainth}. In order to deduce the a priori estimate, in what
follows, we will give some energy estimates. Then, Theorem
\ref{1mainth} is an easy consequence of Lemma \ref{lemma3.1} and
Lemma \ref{lemma3.2}.\par The first lemma is concerned with the
time--decay rate of $\|(\varrho,\sqrt{\rho}({\bf u}-{\bf
m}_0))\|_{L^2(\mathbb T^3)}$.
 \begin{Lemma}\label{lemma3.1}  Under the assumptions of Theorem \ref{1mainth}, there exist two positive constants $C_3>0$ and $\eta_3>0$,
 which are dependent on $M$ and $K$, but independent of $t$, such that
\begin{equation}\label{3.1}\|\varrho(\cdot,t)\|_{L^2(\mathbb T^3)}+\|\sqrt{\rho}({\bf u}-{\bf m}_0)(\cdot,t)\|_{L^2(\mathbb T^3)}
\leq C_3\text{e}^{-\eta_3 t},\end{equation} for any $t\geq 0$.
\end{Lemma}
\begin{proof} We split the proof into three steps.

\texttt{Step 1.  $L^2$ estimate of $(\varrho,{\bf v})$.} Recalling ${\bf v}={\bf u}-{\bf m_0}$, multiplying
the momentum conservation equation \eqref{1.1}$_2$ by ${\bf u}$, and
then integrating the resultant equation over $\mathbb T^3$, we have
from integration by parts that
\begin{equation}\frac{\mathrm{d}}{\mathrm{d}t}\int_{\mathbb T^3}\frac{1}{2}\rho|{\bf u}|^2\mathrm{d}{\bf x}+\int_{\mathbb T^3}\nabla P(\rho){\bf u}\mathrm{d}{\bf x}+\int_{\mathbb T^3}\mu|\nabla
{\bf u}|^2+(\lambda+\mu)|\text{div}{\bf u}|^2\mathrm{d}{\bf
x}=0.\label{3.2}\end{equation} It follows from mass conservation
equation \eqref{1.1}$_1$ and the definition of $G(\rho)$ in
\eqref{2.1} that
$$\left(G(\rho)\right)_t+\text{div}(G(\rho){\bf u})+(P(\rho)- 1)\text{div}{\bf u}=0.$$
Integrating the above equation over $\mathbb T^3$ and then adding the resulting equality to \eqref{3.2}, one has
\begin{equation}\label{3.3}\frac{\mathrm{d}}{\mathrm{d}t}\int_{\mathbb T^3}\frac{1}{2}\rho|{\bf u}|^2+G(\rho)\mathrm{d}{\bf x}+\int_{\mathbb T^3}\mu|\nabla
{\bf u}|^2+(\mu+\lambda)|\text{div}{\bf u}|^2\mathrm{d}{\bf
x}=0.\end{equation} Noticing that
\begin{equation}\label{3.4}\int_{\mathbb T^3}\rho\mathrm{d}{\bf x}=1,~~\text{and}~~\int_{\mathbb T^3}\rho{\bf u}\mathrm{d}{\bf x}={\bf
m}_0,
\end{equation}
we have
\begin{equation}\nonumber\begin{split}\int_{\mathbb T^3}\rho|{\bf v}|^2\mathrm{d}{\bf x}=&\int_{\mathbb T^3}\rho|({\bf u}-{\bf m}_0)|^2\mathrm{d}{\bf x}\\=&\int_{\mathbb T^3}\rho|{\bf u}|^2\mathrm{d}{\bf x}-2\int_{\mathbb T^3}\rho{\bf u}\cdot{\bf m}_0\mathrm{d}{\bf x}+\int_{\mathbb T^3}\rho|{\bf m}_0|^2\mathrm{d}{\bf x}\\=&\int_{\mathbb T^3}\rho|{\bf u}|^2\mathrm{d}{\bf x}-|{\bf m}_0|^2.
\end{split}\end{equation}
Therefore, the equality \eqref{3.3} can be rewritten as follows
\begin{equation}\label{3.5}\frac{\mathrm{d}}{\mathrm{d}t}\int_{\mathbb T^3}\frac{1}{2}\rho|{\bf v}|^2+G(\rho)\mathrm{d}{\bf x}+\int_{\mathbb T^3}\mu|\nabla
{\bf u}|^2+(\mu+\lambda)|\text{div}{\bf u}|^2\mathrm{d}{\bf
x}=0.\end{equation} \texttt{Step 2: Dissipation of $\varrho$.} From
\eqref{1.1}$_1$ and \eqref{1.1}$_2$, we have
\begin{equation}\label{3.6}\begin{split}& (\rho{\bf v})_t+\text{div}(\rho{\bf v}\otimes{\bf v})+\text{div}(\rho{\bf m}_0\otimes{\bf v})+\nabla(P(\rho)-1)=\mu\Delta {\bf v}+(\mu+\lambda) \nabla \hbox{div} {\bf v}.
\end{split}\end{equation}
Applying the operator $\Delta^{-1}\text{div}$ to \eqref{3.6}, one
has
\begin{equation}\label{3.7}  P(\rho)-1=-\partial_t\Delta^{-1}\text{div}(\rho{\bf v})+(2\mu+\lambda)\text{div}{\bf v}-{\mathcal R}_i{\mathcal R}_j(\rho v^iv^j)-{\mathcal R}_i{\mathcal R}_j(\rho m_0^iv^j),
\end{equation}
where ${\mathcal R}_i=-(-\Delta)^{-1/2}\partial_{x_i}$ is the usual Riesz transform on $\mathbb
T^3$. To achieve
the dissipation on $\varrho$, we take the $L^2$ inner product of the
equation \eqref{3.7} with $\varrho$ to get that
\begin{equation}\label{3.8}\begin{split}&\int_{\mathbb T^3} (P(\rho)-1)\varrho\mathrm{d}{\bf x}\\=&-\int_{\mathbb T^3}\partial_t\left[\Delta^{-1}\text{div}(\rho{\bf v})\right]\varrho\mathrm{d}{\bf x}+(2\mu+\lambda)\int_{\mathbb T^3}\text{div}{\bf v}\varrho\mathrm{d}{\bf x}\\&-\int_{\mathbb T^3}{\mathcal R}_i{\mathcal R}_j(\rho v^iv^j)\varrho\mathrm{d}{\bf x}-\int_{\mathbb T^3}{\mathcal R}_i{\mathcal R}_j(\rho m_0^iv^j)\varrho\mathrm{d}{\bf x}\\ \overset{\triangle}=&~I_{11}+I_{12}+I_{13}+I_{14}.
\end{split}\end{equation}
For the  term in the left--side of  \eqref{3.8}, it follows from
Lemma \ref{lemma2.2} that there exists a positive constant $C_4$
such that
\begin{equation}\label{3.9}\int_{\mathbb T^3} (P(\rho)-1)\varrho\mathrm{d}{\bf x}\mathrm{d}{\bf x}\ge C_4 \| \varrho\|^2_{L^2(\mathbb T^3)}.\end{equation}
We turn to estimate each term on the right--side of \eqref{3.8}. For
the term $I_{11}$, it follows from $\eqref{1.1}_1$, \eqref{1.7},
integration by parts, Parseval's theorem, Marcinkiewicz multiplier
theorem and Young's inequality  that
\begin{equation}\label{3.10}\begin{split}I_{11}=&\int_{\mathbb T^3}\partial_t
\left[(-\Delta)^{-\frac{1}{2}}\text{div}(\rho{\bf v})\right](-\Delta)^{-\frac{1}{2}}\varrho\mathrm{d}{\bf x}
\\=&\frac{\mathrm{d}}{\mathrm{d}t}\int_{\mathbb T^3}(-\Delta)^{-\frac{1}{2}}\text{div}(\rho{\bf v})(-\Delta)^
{-\frac{1}{2}}\varrho\mathrm{d}{\bf x}-\int_{\mathbb T^3}(-\Delta)^{-\frac{1}{2}}\text{div}(\rho{\bf v})(-\Delta)^{-\frac{1}{2}}
\varrho_t\mathrm{d}{\bf x}\\=&\frac{\mathrm{d}}{\mathrm{d}t}\int_{\mathbb T^3}(-\Delta)^{-\frac{1}{2}}\text{div}(\rho{\bf v})
(-\Delta)^{-\frac{1}{2}}\varrho\mathrm{d}{\bf x}+\int_{\mathbb T^3}(-\Delta)^{-\frac{1}{2}}\text{div}(\rho{\bf v})
(-\Delta)^{-\frac{1}{2}}\text{div}(\rho{\bf u})\mathrm{d}{\bf x}\\=&\frac{\mathrm{d}}
{\mathrm{d}t}\int_{\mathbb T^3}(-\Delta)^{-\frac{1}{2}}\text{div}(\rho{\bf v})
(-\Delta)^{-\frac{1}{2}}\varrho\mathrm{d}{\bf x}+\int_{\mathbb T^3}\big{|}
(-\Delta)^{-\frac{1}{2}}\text{div}(\rho{\bf v})\big{|}^2\mathrm{d}{\bf x}\\&+
\int_{\mathbb T^3}\big[(-\Delta)^{-\frac{1}{2}}\text{div}(\rho{\bf v})\big]
\big[(-\Delta)^{-\frac{1}{2}}\text{div}(\varrho{\bf m}_0)\big]\mathrm{d}{\bf x}\\
\leq &\frac{\mathrm{d}}{\mathrm{d}t}\int_{\mathbb
T^3}(-\Delta)^{-\frac{1}{2}}\text{div} (\rho{\bf
v})(-\Delta)^{-\frac{1}{2}}\varrho\mathrm{d}{\bf
x}+C\left(\|{\rho}{\bf v}\|_{L^2(\mathbb T^3)}^2 +\|{\rho}{\bf
v}\|_{L^2(\mathbb T^3)}\|\varrho\|_{L^2(\mathbb T^3)}\right)\\ \leq
&-\frac{\mathrm{d}}{\mathrm{d}t} \int_{\mathbb
T^3}\Delta^{-1}\text{div}(\rho{\bf v})\varrho\mathrm{d}{\bf
x}+C\|{\bf v}\|_{L^2(\mathbb
T^3)}^2+\frac{C_4}{6}\|\varrho\|_{L^2(\mathbb T^3)}^2.
\end{split}\end{equation}
 Using Young's inequality, the term $I_{12}$  is controlled as
\begin{equation}\label{3.11}I_{12}\leq C\|{\nabla \bf u}\|_{L^2(\mathbb T^3)}^2+\frac{C_4}{6}\|\varrho\|_{L^2(\mathbb T^3)}^2.
\end{equation}
From \eqref{1.7}, Marcinkiewicz multiplier theorem and Young's
inequality,  the last two terms $I_{13}$ and $I_{14}$ can be bounded
as
\begin{equation}\label{3.12}\begin{split}I_{13}+I_{14}\leq &~C\left(\|{\mathcal R}_i{\mathcal R}_j(\rho v^iv^j)\|_{L^2(\mathbb T^3)}\|\varrho\|_{L^2(\mathbb T^3)}+\|{\mathcal R}_i{\mathcal R}_j(\rho m_0^iv^j)\|_{L^2(\mathbb T^3)}\|\varrho\|_{L^2(\mathbb T^3)}\right)\\ \leq&~ C\left(\|\rho v^iv^j\|_{L^2(\mathbb T^3)}\|\varrho\|_{L^2(\mathbb T^3)}+\|\rho m_0^iv^j\|_{L^2(\mathbb T^3)}\|\varrho\|_{L^2(\mathbb T^3)}\right)\\ \leq&~ C\left(\|\rho\|_{L^\infty(\mathbb T^3)} \|{\bf v}\|_{L^4(\mathbb T^3)}^2\|\varrho\|_{L^2(\mathbb T^3)}+\|\rho|{\bf m}_0|\|_{L^\infty(\mathbb T^3)}\|{\bf v}\|_{L^2(\mathbb T^3)}\|\varrho\|_{L^2(\mathbb T^3)}\right)\\ \leq&~ C\left(\|{\bf v}\|_{L^4(\mathbb T^3)}^2+\|{\bf v}\|_{L^2(\mathbb T^3)}^2\right)+\frac{C_3}{6}\|\varrho\|_{L^2(\mathbb T^3)}^2.
\end{split}\end{equation}
Substituting \eqref{3.9}--\eqref{3.12} into \eqref{3.8}, we obtain
\begin{equation}\label{3.13}\frac{\mathrm{d}}{\mathrm{d}t}\int_{\mathbb T^3}\Delta^{-1}\text{div}(\rho{\bf v})\varrho\mathrm{d}{\bf x}+\frac{C_4}{2}\|\varrho\|_{L^2(\mathbb T^3)}^2\leq C\left(\|{\bf v}\|_{L^4(\mathbb T^3)}^2+\|{\bf v}\|_{L^2(\mathbb T^3)}^2+\|{\nabla \bf u}\|_{L^2(\mathbb T^3)}^2\right).
\end{equation}
\texttt{Step 3: Closing the estimates.} We choose a positive  constant $D_1$ suitably large and define  the temporal energy functional
\begin{equation}\nonumber \mathcal M_1(t)=D_1\left(\int_{\mathbb T^3}\frac{1}{2}\rho|{\bf v}|^2+G(\rho)\mathrm{d}{\bf x}\right)+\int_{\mathbb T^3}
\Delta^{-1}\text{div}(\rho{\bf v})\varrho\mathrm{d}{\bf x},
\end{equation}
for any $t\ge 0$. By virtue of \eqref{1.7}, Lemma \ref{lemma2.2},
H\"older's inequality and Marcinkiewicz's multiplier theorem, we
have
\begin{equation}\nonumber\begin{split}&\Big{|}\int_{\mathbb T^3}\Delta^{-1}\text{div}(\rho{\bf v})\varrho\mathrm{d}{\bf x}\Big{|}\\ \leq &~\|\Delta^{-1}\text{div}(\rho{\bf v})\|_{L^6(\mathbb T^3)}\|\varrho\|_{L^{\frac{6}{5}}(\mathbb T^3)}\\ \leq &~C\|\rho{\bf v}\|_{L^2(\mathbb T^3)}\|\varrho\|_{L^2(\mathbb T^3)}\\ \leq &~ C\left(\int_{\mathbb T^3}\frac{1}{2}\rho|{\bf v}|^2+G(\rho)\mathrm{d}{\bf x}\right).
\end{split}\end{equation}
Thus, $\mathcal M_1(t)$ is  equivalent to
$\|(\varrho,\sqrt{\rho}{\bf v})(t)\|^2_{L^2(\mathbb T^3)}$ if we
choose  $D_1$  large enough.

From \eqref{1.7}, Minkowski's inequality, H\"older's inequality and
Poincar\'e's inequality, we obtain
\begin{equation}\label{3.14}\begin{split}\|{\bf v}\|_{L^r({\mathbb T^3})}\leq &~\|{\bf u}-\bar{\bf u}\|_{L^r({\mathbb T^3})}+|\bar {\bf u}- {\bf m}_0|\\ \leq &~
\|{\bf u}-\bar{\bf u}\|_{L^r({\mathbb T^3})}+\Big{|}\int_{\mathbb T^3}\left(\rho\bar{\bf u}-\rho{\bf u}\right)\mathrm{d}{\bf x}\Big{|}\\ \leq &~
C\|{\bf u}-\bar{\bf u}\|_{L^r({\mathbb T^3})}\\ \leq &~
\begin{cases} C   \|\nabla{\bf u}\|_{L^2({\mathbb T^3})},\  \text{when}\  1\leq r\leq 6,\\
C   \|\nabla{\bf u}\|_{L^3({\mathbb T^3})},\  \text{when}\  1\leq r<\infty.
\end{cases}
\end{split}\end{equation}
Taking a linear combination of \eqref{3.5} and \eqref{3.13} and using \eqref{3.14}, we obtain
\begin{equation}\label{3.15}\frac{\mathrm{d}}{\mathrm{d}t}\mathcal M_1(t)+\frac{\mathcal M_1(t)}{D_1}+\frac{\|\nabla{\bf u}(t)\|^2
_{L^2({\mathbb T^3})}}{D_1}\leq 0,\end{equation} for any $t\ge 0$.
Integrating the above inequality with respect to $t$ over $[0, t]$, \eqref{3.1}
follows immediately. The proof of lemma is completed.
\end{proof}

In the following lemma, we derive the time--decay rate of
$\|\nabla{\bf u}\|_{L^2(\mathbb T^3)}$. The main observation here is
that $\|\nabla{\bf u}(t)\|_{L^2(\mathbb T^3)}$ is sufficiently small
for any large enough $t$.
 \begin{Lemma}\label{lemma3.2} Under the assumptions of Theorem \ref{1mainth}, there exist two positive constants $C_5>0$ and $\eta_4>0$,
 which are dependent on $M$ and $K$, but independent of $t$, such that
\begin{equation}\label{3.16}\|\nabla{\bf u}\|_{L^2(\mathbb T^3)}\leq C_5\text{e}^{-\eta_4 t},\end{equation}
for any $t\ge 0$.\end{Lemma}
\begin{proof} By the definition of material derivative, we can
rewrite \eqref{1.1}$_2$ as follows
\begin{equation}\rho\dot {\bf u}+\nabla(P(\rho)-1)=\mu\Delta {\bf u}+(\mu+\lambda)\nabla \text{div}{\bf u}.\label{3.17}\end{equation}
Multiplying \eqref{3.17} by $\dot {\bf u}$ and then integrating the
resultant equation over $\mathbb T^3$, one has
\begin{equation}\label{3.18}\int_{\mathbb T^3}\rho|\dot {\bf u}|^2\mathrm{d}{\bf x}+\int_{\mathbb T^3}\nabla (P(\rho)-1)\dot {\bf u}\mathrm{d}{\bf x}=\int_{\mathbb T^3}\left(\mu\Delta{\bf u}+(\mu+\lambda)\nabla\text{div}{\bf u}\right)\dot {\bf u}\mathrm{d}{\bf x}.
\end{equation}
Using \eqref{1.1}$_1$ and integration by parts several times, the
second term on the left--side of \eqref{3.18} can be rewritten as
follows:
\begin{equation}\label{3.19}\begin{split}&\int_{\mathbb T^3}\nabla (P(\rho)-1)\dot {\bf u}\mathrm{d}{\bf x}\\=&
\int_{\mathbb T^3}\nabla(P(\rho)-1)\left( {\bf u}_t+{\bf u}\cdot\nabla{\bf u}\right)\mathrm{d}{\bf x}\\=&
-\frac{\mathrm{d}}{\mathrm{d}t}\int_{\mathbb T^3}\left(P(\rho)-1)\right)\text{div}{\bf u}\mathrm{d}{\bf x}+\int_{\mathbb T^3}P'(\rho)\rho_t\text{div}{\bf u}\mathrm{d}{\bf x}\\&+\int_{\mathbb T^3}\left(P'(\rho){\bf u}\cdot\nabla\rho\text{div}{\bf u}+P(|\text{div}{\bf u}|^2-{ u}^i_j{ u}^j_i)\right)\mathrm{d}{\bf x}\\=&
-\frac{\mathrm{d}}{\mathrm{d}t}\int_{\mathbb T^3}\left(P(\rho)-1)\right)\text{div}{\bf u}\mathrm{d}{\bf x}+\int_{\mathbb T^3}\left(-\rho P'(\rho)|\text{div}{\bf u}|^2+P(|\text{div}{\bf u}|^2-{ u}^i_j{ u}^j_i)\right)\mathrm{d}{\bf x}.
\end{split}\end{equation}
Similarly, the term on the right--side  of \eqref{3.18} can be
rewritten as follows:
\begin{equation}\label{3.20}\begin{split}&\int_{\mathbb T^3}\left(\mu\Delta{\bf u}+(\mu+\lambda)\text{div}\nabla{\bf u}\right)\dot {\bf u}\mathrm{d}{\bf x}\\ =&\int_{\mathbb T^3}\left(\mu\Delta{\bf u}+(\mu+\lambda)\nabla\text{div}{\bf u}\right)\left({\bf u}_t+{\bf u}\cdot\nabla{\bf u}\right)\mathrm{d}{\bf x}\\=&-\frac{1}{2}\frac{\mathrm{d}}{\mathrm{d}t}\int_{\mathbb T^3}\left(\mu|\nabla{\bf u}|^2+(\mu+\lambda)|\text{div}{\bf u}|^2\right)\mathrm{d}{\bf x}-\mu\int_{\mathbb T^3}\left({ u}^i_j{ u}^k_j{ u}^i_k-\frac{1}{2}|{ u}^i_j|^2\text{div}{\bf u}\right)\mathrm{d}{\bf x}
\\&-(\mu+\lambda)\int_{\mathbb T^3}\left({ u}^i_j{ u}^j_i\text{div}{\bf u}-\frac{1}{2}(\text{div}{\bf u})^3\right)\mathrm{d}{\bf x}.
\end{split}\end{equation}
Substituting \eqref{3.19} and \eqref{3.20} into \eqref{3.18}, one has
\begin{equation}\label{3.21}\begin{split}&\frac{\mathrm{d}}{\mathrm{d}t}\int_{\mathbb T^3}\left[\frac{1}{2}\left(\mu|\nabla{\bf u}|^2+(\mu+\lambda)|\text{div}{\bf u}|^2\right)-\left(P(\rho)-1\right)\text{div}{\bf u}\right]\mathrm{d}{\bf x}+\int_{\mathbb T^3}\rho|\dot {\bf u}|^2\mathrm{d}{\bf x}\\ =&\int_{\mathbb T^3}\left(\rho P'(\rho)|\text{div}{\bf u}|^2-P(|\text{div}{\bf u}|^2-{ u}^i_j{ u}^j_i)\right)\mathrm{d}{\bf x}-\mu\int_{\mathbb T^3}\left({ u}^i_j{ u}^k_j{ u}^i_k-\frac{1}{2}|{ u}^i_j|^2\text{div}{\bf u}\right)\mathrm{d}{\bf x}
\\&-(\mu+\lambda)\int_{\mathbb T^3}\left({ u}^i_j{ u}^j_i\text{div}{\bf u}-\frac{1}{2}(\text{div}{\bf u})^3\right)\mathrm{d}{\bf x}\\ =&~I_{21}+I_{22}+I_{23}.
\end{split}\end{equation}
For the first term on the right--side of \eqref{3.21}, it follows
from \eqref{1.7} that
\begin{equation}\label{3.22}|I_{21}|\leq C \|\nabla{\bf u}\|^2_{L^2({\mathbb T^3})}.
\end{equation}
For the last two terms on the right--side of \eqref{3.21}, by virtue
of \eqref{1.7}, Lemma \ref{lemma2.2}, Lemma \ref{lemma2.3},
Sobolev's inequality and Young's inequality, we have
\begin{equation}\label{3.23}\begin{split}|I_{22}|+|I_{23}|\leq &~
C\|\nabla{\bf u}\|^3_{L^3({\mathbb T^3})}\\ \leq& ~C\|\nabla{\bf u}\|
^\frac{3}{2}_{L^2({\mathbb T^3})}\|\nabla{\bf u}\|^\frac{3}{2}_{L^6({\mathbb T^3})}\\
 \leq& ~C\|\nabla{\bf u}\|_{L^2({\mathbb T^3})}^\frac{3}{2}\|(F,\nabla{\bf w},P(\rho)-1)
 \|^\frac{3}{2}_{L^6({\mathbb T^3})}\\ \leq& ~C\|\nabla{\bf u}\|^\frac{3}{2}_{L^2({\mathbb T^3})}
 \left(\|(\nabla F,\nabla^2{\bf w})\|^\frac{3}{2}_{L^2({\mathbb T^3})}+\| F\|^\frac{3}{2}
 _{L^2({\mathbb T^3})}+\|(P(\rho)-1)\|^\frac{3}{2}_{L^6({\mathbb T^3})}\right)\\
  \leq& ~C\|\nabla{\bf u}\|^\frac{3}{2}_{L^2({\mathbb T^3})}\left(\|\sqrt{\rho}\dot{\bf u}
  \|^\frac{3}{2}_{L^2({\mathbb T^3})}+\|\varrho\|^\frac{3}{2}_{L^2({\mathbb T^3})}
  +\|\nabla{\bf u}\|^\frac{3}{2}_{L^2({\mathbb T^3})}+\|\varrho\|^\frac{3}{2}_{L^6({\mathbb T^3})}\right)\\
  \leq& ~\frac{1}{2}\|\sqrt{\rho}\dot{\bf u}\|^2_{L^2({\mathbb T^3})}+C\left(\|\nabla{\bf u}\|^6
  _{L^2({\mathbb T^3})}+\|\nabla{\bf u}\|^3_{L^2({\mathbb
  T^3})}\right.\\
  &\quad\left.+\|\varrho\|^6_{L^2({\mathbb T^3})}+\|\nabla{\bf
u}\|^2_{L^2({\mathbb T^3})}
  +\|\varrho\|^2_{L^6({\mathbb T^3})}\right)\\ \leq& ~\frac{1}{2}
  \|\sqrt{\rho}\dot{\bf u}\|^2_{L^2({\mathbb T^3})}+C\left(\|\nabla{\bf u}\|^6_{L^2({\mathbb T^3})}
  +\|\nabla{\bf u}\|^3_{L^2({\mathbb T^3})}+\|\nabla{\bf u}\|^2_{L^2({\mathbb T^3})}+\|\varrho\|^2_{L^2({\mathbb T^3})}\right).
\end{split}\end{equation}
Plugging  \eqref{3.22}--\eqref{3.23} into \eqref{3.21}, we have
\begin{equation}\label{3.24}\begin{split}&\frac{\mathrm{d}}{\mathrm{d}t}
\int_{\mathbb T^3}\left[\frac{1}{2}\left(\mu|\nabla{\bf u}|^2+(\mu+\lambda)
|\text{div}{\bf u}|^2\right)-\left(P(\rho)-1\right)\text{div}{\bf u}\right]
\mathrm{d}{\bf x}+\frac{1}{2}\int_{\mathbb T^3}\rho|\dot {\bf u}|^2\mathrm{d}{\bf x}
\\ \leq &~C\left(\|\nabla{\bf u}\|^6_{L^2({\mathbb T^3})}
+\|\nabla{\bf u}\|^3_{L^2({\mathbb T^3})}+\|\nabla{\bf u}\|^2_{L^2({\mathbb T^3})}+\|\varrho\|^2_{L^2({\mathbb T^3})}\right).
\end{split}\end{equation}
This, combined with  \eqref{1.6}, \eqref{1.7}, \eqref{3.3} and \eqref{lemma3.1} yields
\begin{equation}\nonumber\sqrt{\rho}\dot{\bf u}\in L^2_{\text{loc}}((0,\infty);L^2(\mathbb T^3)),\end{equation} and
 \begin{equation}\nonumber\int_{\mathbb T^3}\left[\frac{1}{2}\left(\mu|\nabla{\bf u}|^2+(\mu+\lambda)|\text{div}{\bf u}|^2\right)-\left(P(\rho)-1\right)\text{div}{\bf u}\right](t)\mathrm{d}{\bf x}\in C[0,\infty),
\end{equation}
which together with \eqref{1.6} implies that
 \begin{equation}\label{3.25}\int_{\mathbb T^3}\left[\frac{1}{2}\left(\mu|\nabla{\bf u}|^2+(\mu+\lambda)
 |\text{div}{\bf u}|^2\right)-\left(P(\rho)-1\right)\text{div}{\bf u}+D_2|\varrho|^2\right](t)\mathrm{d}{\bf x}\in C[0,\infty),
\end{equation}
where $D_2$ is a suitably large positive constant. In light of
\eqref{3.1} and \eqref{3.5}, we obtain
 \begin{equation}\label{3.26}\int_0^\infty\int_{\mathbb T^3}\left[\frac{1}{2}\left(\mu|\nabla{\bf u}|^2+(\mu+\lambda)|\text{div}{\bf u}|^2\right)-\left(P(\rho)-1\right)\text{div}{\bf u}+D_2|\varrho|^2\right]\mathrm{d}{\bf x}\mathrm{d}t<\infty.
\end{equation}
Next, we choose a positive  constant $D_3$ suitably large and define
the temporal energy functional
\begin{equation}\nonumber \mathcal M_2(t)=D_3\mathcal M_1(t)+\int_{\mathbb T^3}\left[\frac{1}{2}
\left(\mu|\nabla{\bf u}|^2+(\mu+\lambda)|\text{div}{\bf
u}|^2\right)-\left(P(\rho)-1\right)\text{div}{\bf
u}+D_2|\varrho|^2\right](t)\mathrm{d}{\bf x},
\end{equation}
for any $t\ge 0$. Note that  $\mathcal M_2(t)$ is  equivalent to
$\|(\varrho,\sqrt{\rho}{\bf v},\nabla{\bf u})(t)\|^2_{L^2(\mathbb
T^3)}$ if we choose  $D_2$ and $D_3$  large enough. Fix a positive
constant $\delta_1$ that may be small. Then, it follows from
\eqref{3.1} and \eqref{3.26} that there exists a positive constant
$T_1>0$ such that
\begin{equation}\label{3.27}\mathcal M_2(T_1)<\delta_1.
\end{equation}
Now, we claim that
\begin{equation}\label{3.28}\int_{\mathbb T^3}\left[\frac{1}{2}\left(\mu|\nabla{\bf u}|^2+(\mu+\lambda)|\text{div}{\bf u}|^2\right)-\left(P(\rho)-1\right)\text{div}{\bf u}+D_2|\varrho|^2\right](t)\mathrm{d}{\bf x}<2\delta_1
\end{equation}
holds for any $t\geq T_1$. Assume this claim for the moment.  Then,
\eqref{3.28} implies that
\begin{equation}\label{3.29}\int_{\mathbb T^3}
\left[\left(\mu|\nabla{\bf
u}|^2+(\mu+\lambda)|\text{div}{\bf
u}|^2\right)\right](t)\mathrm{d}{\bf x}<4\delta_1,
\end{equation}
 for any $t\geq T_1$. Let $\delta_1$ be small enough, then taking a linear combination of \eqref{3.15} and \eqref{3.24} yields
\begin{equation}\label{3.30}\frac{\mathrm{d}}{\mathrm{d}t}\mathcal M_2(t)+\frac{\mathcal M_2(t)}{D_3}+\frac{\|\sqrt{\rho}\dot{\bf u}(t)\|
^2_{L^2({\mathbb T^3})}}{D_3}\leq 0,\end{equation} for any $t\geq
T_1$. Integrating \eqref{3.30} with respect to $t$ over $[0, t]$ gives
\eqref{3.16} immediately. Thus, to complete the proof of Lemma
\ref{lemma3.2}, it suffices to establish \eqref{3.28}.

Next, we return to the proof of \eqref{3.28}. If \eqref{3.28} is
false, by \eqref{3.25}, there exists a time $T_2> T_1$ such that
\begin{equation}\label{3.31}\int_{\mathbb T^3}\left[\frac{1}{2}\left(\mu|\nabla{\bf u}|^2
+(\mu+\lambda)|\text{div}{\bf u}|^2\right)-\left(P(\rho)-1\right)\text{div}{\bf u}+D_2|\varrho|^2\right](T_2)\mathrm{d}{\bf x}=2\delta_1.
\end{equation}
Taking a minimal value of $T_2$ satisfying \eqref{3.31}, then \eqref{3.28} holds for any $T_1\le t<T_2$. Integrating \eqref{3.30} from $T_1$ to $T_2$, one has $$\mathcal M_2(T_2)\leq \mathcal M_2(T_1)<\delta_1,$$
which contradicts \eqref{3.31}. Hence \eqref{3.28} holds for any $t\ge T_1$. The proof of lemma is completed.
\end{proof}

\bigskip
\bigskip
\setcounter{equation}{0}
\section{Proof of Theorem \ref{2mainth} and Theorem \ref{3mainth}}
\setcounter{equation}{0}
\smallskip

We turn to prove Theorem \ref{2mainth} and Theorem \ref{3mainth} in
this section. The following lemma is devoted to deriving uniform
positive lower bound of $\rho$.
\begin{Lemma}\label{lemma4.1}  Under the assumptions of Theorem \ref{2mainth}, there exists a
 positive constant $c_1>0$, which is independent of $t$,  such that
\begin{equation}\label{4.1}\inf\limits_{{\bf x}\in\mathbb T^3}\rho({\bf x},t)\ge c_1,\end{equation}
for any $t\geq 0$.
\end{Lemma}
\begin{proof}  First, motivated by Desjardins \cite{Desjardins}, we rewrite mass conservation equation \eqref{1.1}$_1$ as
\begin{equation}\label{4.2}(\log\rho)_t+{\bf u}\cdot\nabla\log\rho+\text{div}{\bf u}=0.
\end{equation}
Defining
$H\overset{\triangle}=(2\mu+\lambda)\log\rho+\Delta^{-1}\text{div}(\rho{\bf
v})$, and then combining \eqref{3.7} with \eqref{4.2}, we have
\begin{equation}\label{4.3} H_t+{\bf u}\cdot\nabla H+(P(\rho)-1)=[v^j,\mathcal R_i\mathcal R_j](\rho v^i)+[m_0^j,\mathcal R_i\mathcal R_j](\rho v^i),
\end{equation}
where $[u^j,\mathcal R_i\mathcal R_j]( v^i)=u^j\mathcal R_i\mathcal
R_j( v^i)-\mathcal R_i\mathcal R_j( v^iu^j)$. Let ${\bf y}\in\mathbb
T^3$ and define the corresponding particle path ${\bf x}(t,{\bf y})$ by
\begin{equation}\nonumber
\begin{cases}
\frac{ \mathrm{d}{\bf x}(t,{\bf y})}{\mathrm{d}t}={\bf u}({\bf x}(t,{\bf y}),t),\\
{\bf x}(t_0,{\bf y})={\bf y}.
\end{cases}
\end{equation}
Then, \eqref{4.3} can be reformulated as
\begin{equation}\label{4.4} \frac{\mathrm{d}}{\mathrm{d}t}H(t)+(P(\rho)-1)=[v^j,\mathcal R_i\mathcal R_j]
(\rho v^i)+[m_0^j,\mathcal R_i\mathcal R_j](\rho v^i).
\end{equation}
In virtue of the results of Coifman, Lions, Meyer and Semmes \cite{Coifman}, the following map
\begin{equation}\nonumber\begin{split}W^{1,r_1}(\mathbb T^{N})^N\times L^{r_2}(\mathbb T^{N})^N&\rightarrow W^{1,r_3}(\mathbb T^{N})^N,\\
({\bf u},{\bf v})&\rightarrow [u_j,\mathcal R_i\mathcal R_j]v_i,
\end{split}\end{equation}
is continuous for any $N\geq 2$ as soon as
$\frac{1}{r_3}=\frac{1}{r_1}+\frac{1}{r_2}$. Hence,    using
\eqref{1.7}, Lemma \ref{lemma2.3} and \eqref{3.14}, we can deduce
that
\begin{equation}\label{4.5}\begin{split}&\big{\|}[v^j,\mathcal R_i\mathcal R_j](\rho v^i)\big{\|}_{L^\infty(\mathbb T^{3})}+\big{\|}[m_0^j,\mathcal R_i\mathcal R_j](\rho v^i)\big{\|}_{L^\infty(\mathbb T^{3})}\\ \leq &~C\left(\big{\|}[v^j,\mathcal R_i\mathcal R_j](\rho v^i)\big{\|}_{W^{1,4}(\mathbb T^{3})}+\big{\|}[m_0^j,\mathcal R_i\mathcal R_j](\rho v^i)\big{\|}_{W^{1,4}(\mathbb T^{3})}\right)\\ \leq &~C\left(
\|{\bf v}\|_{W^{1,6}(\mathbb T^{3})}+\|{\bf m}_0\|_{W^{1,6}(\mathbb T^{3})}\right)\|\rho{\bf v}\|_{L^{12}(\mathbb T^{3})}\\ \leq &~C\left(
\|\nabla {\bf u}\|_{L^{6}(\mathbb T^{3})}+1\right)\|{\bf v}\|_{L^{12}(\mathbb T^{3})}\\ \leq &~C\left(
\|F\|_{L^6(\mathbb T^{3})}+\|\nabla{\bf w}\|_{L^6(\mathbb T^{3})}+\|(P(\rho)-1)\|_{L^6(\mathbb T^{3})}+1\right)\|\nabla{\bf u}\|_{L^{3}(\mathbb T^{3})}\\ \leq &~C\left(
\|\sqrt{\rho}\dot{\bf u}\|_{L^2(\mathbb T^{3})}+1\right)\|\nabla{\bf u}\|^{\frac{1}{2}}_{L^{2}(\mathbb T^{3})}\|\nabla{\bf u}\|^{\frac{1}{2}}_{L^{6}(\mathbb T^{3})}\\
 \leq &~C\left(
\|\sqrt{\rho}\dot{\bf u}\|_{L^2(\mathbb
T^{3})}+1\right)^{\frac{3}{2}}\|\nabla{\bf
u}\|^{\frac{1}{2}}_{L^{2}(\mathbb T^{3})}.
\end{split}\end{equation}
On the other hand, it follows from \eqref{3.30} that
\begin{equation}\label{4.6}\int_0^\infty\int_{\mathbb T^3}\rho|\dot{\bf u}|^2\mathrm{d}{\bf x}\mathrm{d}t<\infty,
\end{equation}
where we have used \eqref{3.16}.
Therefore, this together with Theorem~\ref{1mainth} and \eqref{4.5}
implies that
\begin{equation}\label{4.7}\begin{split}&\int_0^\infty\big{\|}[v^j,\mathcal R_i\mathcal R_j](\rho v^i)\big{\|}_{L^\infty(\mathbb T^{3})}\mathrm{d}t+\int_0^\infty\big{\|}[m_0^j,\mathcal R_i\mathcal R_j](\rho v^i)\big{\|}_{L^\infty(\mathbb T^{3})}\mathrm{d}t<\infty.
\end{split}\end{equation}
In virtue  \eqref{1.8}, \eqref{4.4} and \eqref{4.7}, it is clear
that
\begin{equation}\label{4.8}H(t)\in C[0,\infty),
\end{equation}
where we have abbreviated $H({\bf x}(t),t)$ by $H(t)$ for
convenience. By virtue of  \eqref{1.7} and \eqref{3.14}, and
Theorem~\ref{1mainth}, one has
\begin{equation}\label{4.9}\|\Delta^{-1}\text{div}(\rho{\bf v})\|_{L^\infty(\mathbb T^{3})}\leq C\|\rho{\bf v}\|_{L^4(\mathbb T^{3})}\leq C\|{\bf v}\|_{L^4(\mathbb T^{3})}\leq C\|\nabla{\bf u}\|_{L^2(\mathbb T^{3})}\leq C\text{e}^{-\eta_1t}.
\end{equation}
Fix a positive constant $\delta_2$ that may be small, in view of \eqref{4.7} and \eqref{4.9}, there exists a positive constant $T_3>0$ such that
\begin{equation}\label{4.10}\begin{split}\int_{t}^\infty\Big{\|}\left([v^j,\mathcal R_i\mathcal R_j](\rho v^i),[m_0^j,
\mathcal R_i\mathcal R_j](\rho v^i)\right)\Big{\|}_{L^\infty(\mathbb
T^{3})}\mathrm{d}t+\|\Delta^{-1}\text{div}(\rho{\bf v})(t)
\|_{L^\infty(\mathbb T^{3})} \leq \delta_2,
\end{split}\end{equation}
for any $t\ge T_3$. Combining \eqref{4.8} and \eqref{4.9}, we see
that $\|\log\rho({\bf x},t)\|_{L^\infty(0,T_3;L^\infty(\mathbb
T^3))}\le C(T_3)$. Assume that there exists a time $T_4\ge T_3$ such
that $0<c_2=\rho(T_4)\le \frac{1}{\text{e}^3}$. Otherwise, we prove
 \eqref{4.1}. Setting
$\kappa=-((2\mu+\lambda)\log\rho+\Delta^{-1}\text{div}(\rho{\bf
v}))(T_4)$, then it is clear that $\kappa>2\mu+\lambda$  if
$\delta_2$ is small enough. Now, we claim that
\begin{equation}\label{4.11}-((2\mu+\lambda)\log\rho+\Delta^{-1}\text{div}(\rho{\bf v}))(t)< 2\kappa
\end{equation}
holds for any $t\ge T_4$. Assume this claim for the moment, then
\eqref{4.1} follows immediately. Next, we return to the proof of
\eqref{4.11}. If \eqref{4.11} is false, by \eqref{4.8}, there exists
a time $T_6> T_4$ such that
\begin{equation}\label{4.12}-((2\mu+\lambda)\log\rho+\Delta^{-1}\text{div}(\rho{\bf v}))(T_6)=2\kappa.
\end{equation}
We take  a minimal value of $T_6$ satisfying \eqref{4.12} and then
choose a maximal value of $T_5<T_6$ such that $-((2\mu+\lambda)\log\rho+\Delta^{-1}\text{div}(\rho{\bf v}))(T_5)= \kappa$. Thus we have
\begin{equation}\label{4.13}-((2\mu+\lambda)\log\rho+\Delta^{-1}\text{div}(\rho{\bf v}))(t)\in [\kappa,2\kappa].
\end{equation}
for any $t\in [T_5,T_6]$, which implies that
$0<\rho(t)<\frac{1}{\text{e}}$ for  any $t\in [T_5,T_6]$. Using
\eqref{4.10}, and integrating \eqref{4.4} along particle
trajectories from $T_5$ to $T_6$, we have
$$-\kappa\ge -\int^{T_5}_{T_6}[P(\rho(t))-1 ]\mathrm{d}t-\int_{T_5}^{T_6}\Big{\|}\left([v^j,\mathcal R_i\mathcal R_j](\rho v^i),[m_0^j,\mathcal R_i\mathcal R_j](\rho v^i)\right)(t)\Big{\|}_{L^\infty(\mathbb T^{3})}\mathrm{d}t\ge -\delta_2,$$
which is impossible if  $\delta_2$ is small enough. We therefore
conclude that there is no such time $T_6$, which is bigger than $T_4$,  such that
$-((2\mu+\lambda)\log\rho+\Delta^{-1}\text{div}(\rho{\bf v}))(T_6)=2\kappa
$. Since ${\bf y}\in\mathbb T^3$ is
arbitrary, we have  $((2\mu+\lambda)\log\rho+\Delta^{-1}\text{div}(\rho{\bf v}))(t)>-2\kappa
$ on $\mathbb
T^3\times[T_4,\infty)$, and \eqref{4.1} follows immediately. The proof of lemma is completed.

\end{proof}

Now we are in a position to prove Theorem~\ref{2mainth} and Theorem~\ref{3mainth}.
\bigskip

\noindent{\bf Proof of Theorem~\ref{2mainth}.} Multiplying \eqref{4.4} by $H(t)$, we have
 \begin{equation}\nonumber \begin{split} &\frac{1}{2}\frac{\mathrm{d}}{\mathrm{d}t}H^2(t)+\frac{P(\rho)-1}{(2\mu+\lambda)\log\rho}H^2(t)\\=&\left([v^j,\mathcal R_i\mathcal R_j]
(\rho v^i)+[m_0^j,\mathcal R_i\mathcal R_j](\rho v^i)+\frac{(P(\rho)-1)\Delta^{-1}\text{div}(\rho{\bf v})}{(2\mu+\lambda)\log\rho}\right)H(t).
\end{split}\end{equation}
 In virtue of \eqref{1.7} and \eqref{4.1}, we see  that  $\log\rho\sim P(\rho)-1$. Hence, there exists a positive constant $\eta_5$ such that
\begin{equation}\nonumber\begin{split} &\frac{\mathrm{d}}{\mathrm{d}t}H^2(t)+\eta_5H^2(t)\\ \le &~C\Big{\|}\left([v^j,\mathcal R_i\mathcal R_j](\rho v^i),[m_0^j,\mathcal R_i\mathcal R_j](\rho v^i),\Delta^{-1}\text{div}(\rho{\bf v})\right)(t)\Big{\|}_{L^\infty(\mathbb T^{3})}|H(t)|,
\end{split}\end{equation}
which implies that
\begin{equation}\label{4.14}\begin{split} \frac{\mathrm{d}}{\mathrm{d}t}|H(t)|+\eta_5|H(t)| \le C\Big{\|}\left([v^j,\mathcal R_i\mathcal R_j](\rho v^i),[m_0^j,\mathcal R_i\mathcal R_j](\rho v^i),\Delta^{-1}\text{div}(\rho{\bf v})\right)(t)\Big{\|}_{L^\infty(\mathbb T^{3})}.
\end{split}\end{equation}
Combining  \eqref{1.8}, \eqref{4.5} and \eqref{4.9} yields
$$\Big{\|}\left([v^j,\mathcal R_i\mathcal R_j](\rho v^i),[m_0^j,\mathcal R_i\mathcal R_j](\rho v^i),\Delta^{-1}\text{div}(\rho{\bf v})\right)(t)\Big{\|}_{L^\infty(\mathbb T^{3})}\le C\left(
\|\sqrt{\rho}\dot{\bf u}\|_{L^2(\mathbb T^{3})}+1\right)^{\frac{3}{2}}\text{e}^{-\frac{\eta_1}{2}t}.$$
Substituting the above estimate into \eqref{4.14}, we obtain
\begin{equation}\nonumber \frac{\mathrm{d}}{\mathrm{d}t}|H(t)|+\eta_5|H(t)|\le C\left(
\|\left(\sqrt{\rho}\dot{\bf u}\right)(t)\|_{L^2(\mathbb T^{3})}+1\right)^{\frac{3}{2}}\text{e}^{-\frac{\eta_1}{2}t},
\end{equation}
which implies
\begin{equation}\label{4.15} \frac{\mathrm{d}}{\mathrm{d}t}\left(\text{e}^{{\eta_5}t}|H(t)|\right)\le C\text{e}^{\eta_5t}\left(
\|\left(\sqrt{\rho}\dot{\bf u}\right)(t)\|_{L^2(\mathbb T^{3})}+1\right)^{\frac{3}{2}}\text{e}^{-\frac{\eta_1}{2}t}.
\end{equation}
Integrating \eqref{4.15} along particle trajectories from $0$ to $t$, and using \eqref{4.6} and H\"older's inequality, we obtain
\begin{equation}\nonumber\begin{split}|H(t)|\le &~\text{e}^{-{\eta_5}t}+C\int_0^t\text{e}^{-{\eta_5}(t-\tau)}\left(
\|\left(\sqrt{\rho}\dot{\bf u}\right)(\tau)\|_{L^2(\mathbb
T^{3})}+1\right)^{\frac{3}{2}}\text{e}^{-\frac{\eta_1}{2}\tau}\mathrm{d}\tau\\
\le
&~\text{e}^{-{\eta_5}t}+C\int_0^\frac{t}{2}\text{e}^{-{\eta_5}(t-\tau)}\left(
\|\left(\sqrt{\rho}\dot{\bf u}\right)(\tau)\|_{L^2(\mathbb
T^{3})}+1\right)^{\frac{3}{2}}\text{e}^{-\frac{\eta_1}{2}\tau}\mathrm{d}\tau\\
&+C\int_\frac{t}{2}^t\text{e}^{-{\eta_5}(t-\tau)}\left(
\|\left(\sqrt{\rho}\dot{\bf u}\right)(\tau)\|_{L^2(\mathbb
T^{3})}+1\right)^{\frac{3}{2}}\text{e}^{-\frac{\eta_1}{2}\tau}\mathrm{d}\tau\\
\le &~\text{e}^{-{\eta_5}t}+C\text{e}^{-\frac{\eta_5
t}{2}}\left(\int_0^\frac{t}{2} \|\left(\sqrt{\rho}\dot{\bf
u}\right)(\tau)\|^2_{L^2(\mathbb
T^{3})}\mathrm{d}\tau\right)^\frac{3}{4}\left(\int_0^\frac{t}{2}\text{e}^{-2\eta_1\tau}\mathrm{d}\tau\right)^\frac{1}{4}
\\&+C\text{e}^{-\frac{\eta_5t}{2}}\int_0^\frac{t}{2}\text{e}^{-\frac{\eta_1}{2}\tau}\mathrm{d}\tau
+C\text{e}^{-\frac{\eta_1t}{4}}\int_\frac{t}{2}^t\text{e}^{-{\eta_5}(t-\tau)}\mathrm{d}\tau\\&+C\text{e}^{-\frac{\eta_1t}{4}}\left(\int_\frac{t}{2}^t
\|\left(\sqrt{\rho}\dot{\bf u}\right)(\tau)\|_{L^2(\mathbb
T^{3})}^2\mathrm{d}\tau\right)^\frac{3}{4}\left(\int_\frac{t}{2}^t\text{e}^{-4{\eta_5}(t-\tau)}\mathrm{d}\tau\right)^\frac{1}{4}\\
\le
&~C\left(\text{e}^{-\frac{\eta_1t}{4}}+\text{e}^{-\frac{\eta_5t}{2}}\right)
\end{split}\end{equation}
as claimed in  \eqref{1.9}. We complete the proof of Theorem~\ref{2mainth}.\hfill$\Box$

\bigskip

\noindent{\bf Proof of Theorem~\ref{3mainth}.}  If the conclusion in Theorem~\ref{3mainth} is false, then there exists a  time
$T_7$ such that $\inf\limits_{{\bf x\in\mathbb T^3}}\rho({\bf x},T_7)>0$. Due to Theorem~\ref{2mainth}, one deduces that
\begin{equation}\nonumber\lim\limits_{t\rightarrow\infty}\|(\rho-1)(\cdot,t)\|_{L^\infty(\mathbb T^3)}=0,
\end{equation}
which implies that there exists a time $T_8(\geq T_7)$, such that for a.e. ${\bf x}\in\mathbb T^3$,
\begin{equation}\label{4.16}\frac{1}{2}\leq \rho({\bf x},T_8)\leq \frac{3}{2}.
\end{equation}
By virtue of  \eqref{1.6}, we see that \begin{equation}\label{4.17}\rho(t)\in C[0,\infty),
\end{equation}
where we have abbreviated $\rho({\bf x}(t),t)$ by $\rho(t)$ for
convenience. Due to \eqref{1.4} and
$\inf\limits_{{\bf x}\in\mathbb T^3}\rho_0({\bf x})=0$, it is clear
that for any $\varepsilon>0$,
 there exists a positive constant  $T_\varepsilon<T_8$  such that
\begin{equation}\label{4.18}\inf\limits_{{\bf x}\in\mathbb T^3}\rho({\bf x},T_\varepsilon)=\varepsilon\quad \text{and}\quad \inf\limits_{t\in[T_\varepsilon,T_8]}\rho(t)\ge\varepsilon.
\end{equation}
Therefore, there exists a
non--zero measurable $\mathcal A$ such that
\begin{equation}\varepsilon\le \rho({\bf y}(T_\varepsilon),T_\varepsilon)\le 2\varepsilon,\label{4.19}
 \end{equation}
 for any ${\bf y}\in \mathcal A$ if $\varepsilon$ is sufficiently small. Integrating \eqref{4.4} along particle trajectories from $T_\varepsilon$ to $T_8$, and using \eqref{1.7}, \eqref{4.7}, \eqref{4.9} and \eqref{4.19}, we have
\begin{equation}\log\rho({\bf x}(T_8,{\bf y}),T_8)\le \log\rho({\bf y}(T_\varepsilon),T_\varepsilon)+C(M,T_8)\le \log(2\varepsilon)+C(M,T_8),
\end{equation}
which contradicts \eqref{4.16} if $\varepsilon$ is small enough. This completes the proof of Theorem~\ref{3mainth}.
 \hfill$\Box$

\bigskip

\section*{Acknowledgments}

Guochun Wu's research was partially supported by National Natural Science Foundation of China
$\#$12271114, and Natural
Science Foundation of Fujian Province $\#$ 2022J01304. L. Yao's research  was
partially  supported by National Natural Science Foundation of China
$\#$12171390, $\#$11931013, and Natural Science Basic Research Plan
for Distinguished Young Scholars in Shaanxi Province of China (Grant
No. 2019JC-26). Yinghui Zhang' research is partially supported by National Natural Science Foundation of China
$\#$12271114, $\#$12001189, and
Guangxi Natural Science Foundation $\#$2019JJG110003,
$\#$2019AC20214.

\end{document}